\documentclass[a4paper,11pt]{amsart}
\usepackage{amssymb}
\usepackage{latexsym}
\usepackage{amsmath}
\usepackage{enumerate}
\usepackage{amsmath, hyperref}

\usepackage{tikz}
\usepackage{geometry}
\usepackage[all,pdf]{xy}

\newtheorem{theorem}{Theorem}[section]
\newtheorem{lemma}[theorem]{Lemma}
\newtheorem{corollary}[theorem]{Corollary}
\newtheorem{proposition}[theorem]{Proposition}

\theoremstyle{definition}
\newtheorem{definition}[theorem]{Definition}
\newtheorem{question}[theorem]{Question}
\newtheorem{fact}[theorem]{Fact}
\newtheorem{example}[theorem]{Example}

\newtheorem*{remark}{Remark}
\newtheorem*{claim}{Claim}

\title{Effective inseparability and some applications in meta-mathematics}
\author{Yong Cheng}

\address{School of Philosophy, Wuhan University, China. ORCID ID: 0000-0003-2408-3886}

\email{world-cyr@hotmail.com}

\thanks{This is a contributed paper for the conference ``Celebrating 90 Years of G\"{o}del's Incompleteness Theorems" organized by Carl Friedrich von Weizs\"{a}cker Center and Kurt G\"{o}del Society. I would like to thank the organizers for making this great conference. I also thank the two referees for helpful comments and suggestions for improvement.}



\begin{document}
\begin{abstract}
Effectively inseparable  pairs and their properties play an important role in the meta-mathematics of arithmetic and incompleteness. Different notions are introduced and shown in the literature to be equivalent to effective inseparability. We give a much simpler proof of these equivalences
using the strong double recursion theorem.
Then we prove some results about the application of effective inseparability in meta-mathematics.
\end{abstract}

\maketitle

\section{Introduction}

Since G\"{o}del, research on incompleteness has greatly deepened our understanding of the incompleteness phenomenon. In particular, Smullyan's work in \cite{GIT, RM}  provides a unique
way to understand incompleteness in an abstract way via metamathematical
research of formal systems. This paper is inspired by Smullyan's work.

Recursion-theoretic proofs of metamathematical results tend to rely on an effectively inseparable  pair of recursively enumerable (RE) \footnote{Some authors use ``computable" instead of ``recursive" in the literature.} sets and its properties.
Effectively inseparable sets arise naturally in the meta-mathematics of arithmetic. For example, the pair of G\"{o}del numbers of provable and refutable sentences of $\mathbf{PA}$ is effectively inseparable.
The motivation of this paper is to study uniform versions of
incompleteness/undecidability via the notion of effective inseparability. Especially, we study effectively inseparable theories that exhibit similar behaviors connected to
incompleteness/undecidability.

In this paper, a theory is an RE theory of classical first-order logic
in finite signature, and we identify a theory  with the set of sentences provable in it. We always assume the \emph{arithmetization} of the base theory and we will usually work with a bijective G\"{o}del numbering of the sentences. Under arithmetization, we equate a set  of sentences  with the set of G\"{o}del's numbers of sentences. For any formula $\phi$, we use $\ulcorner\phi\urcorner$ to denote the G\"{o}del number of $\phi$. We first introduce the standard notions of recursively inseparable and effectively inseparable  pairs of RE sets.

\begin{definition}[\cite{Rogers}]\label{}
Let $(A, B)$ be a disjoint pair  of RE sets.
\begin{enumerate}[(1)]
  \item We say $(A, B)$ is  \emph{recursively inseparable} ($\sf RI$) if there exists no recursive superset of $A$ which is disjoint from $B$.
  \item We say $(A, B)$  is \emph{effectively inseparable} $(\sf EI)$ if there is a recursive function $f(x,y)$ such that for any $i$ and $j$, if $A\subseteq W_i$ and $B\subseteq W_j$ with $W_i\cap W_j=\emptyset$, then $f(i,j)\notin W_i\cup W_j$.
\end{enumerate}
\end{definition}

Effective inseparability can be viewed as an effective version of recursive inseparability. For a  disjoint pair $(A, B)$ of RE sets, if $(A, B)$ is $\sf RI$, then for any $i$ and $j$, if  $A\subseteq W_i, B\subseteq W_j$ and $W_i\cap W_j=\emptyset$, then $W_i\cup W_j\neq\mathbb{N}$; if $(A, B)$ is $\sf EI$, then we can effectively pick an element not in $W_i\cup W_j$.

In Definition \ref{The nuclei of a theory},  we  introduce the notions of recursively inseparable theories and effectively inseparable theories which are based on the nuclei of a theory.

\begin{definition}[The nuclei of a theory, Smullyan]\label{The nuclei of a theory}
Let $T$ be a consistent RE theory.
\begin{enumerate}[(1)]
  \item The pair $(T_P, T_R)$ are called the \emph{nuclei} of the theory $T$, where  $T_P$ is the set of G\"{o}del numbers of sentences provable in $T$, and $T_R$ is the set of G\"{o}del numbers of sentences refutable in $T$. In other words, $T_P=\{\ulcorner\phi\urcorner: T\vdash\phi\}$, and $T_R=\{\ulcorner\phi\urcorner: T\vdash\neg\phi\}$.
  \item We say $T$ is $\sf RI$  if $(T_P, T_R)$ is recursively inseparable.
      \item We say $T$ is $\sf EI$ if $(T_P, T_R)$ is effectively inseparable.
\end{enumerate}
\end{definition}

The nuclei of a theory play an important role in  meta-mathematical
research on incompleteness and undecidability.
The notion of an $\sf EI$ theory is stronger than that of a $\sf RI$
theory. A $\sf RI$ theory may not be an $\sf EI$ theory (see Section 3), and
$\sf EI$ theories have nice properties. The notion of an $\sf EI$ theory is
central in research on the incompleteness phenomenon (see \cite{Smullyan, Smullyan63,
RM}). For example, if $T$ is a consistent $\sf EI$ theory, then there is a recursive function $f$ such that if $T_P=W_i$ and $T_R=W_j$, then $f(i,j)$ converges and outputs the code of a sentence $\phi$ which is independent of $T$ (i.e., $T\nvdash\phi$ and $T\nvdash\neg\phi$). We could view effective inseparability as the effective version of essential incompleteness.

Smullyan introduces different notions (see Definition \ref{def of EI}) and essentially shows in \cite{RM} that all these notions are equivalent to the notion of effective inseparability. These equivalences reveal the central role of effective inseparability in the meta-mathematics of arithmetic. We will show in Section 2 that these equivalences can be proved in a much simpler
and more efficient way using the strong double recursion theorem.

The structure of this paper is as follows. In Section 1, we introduce the notion of
effectively inseparable theories and the motivation of this paper. In Section 2, we study alternative characterizations of effectively inseparable pairs and establish these  equivalences in a much simpler
and more efficient way than the proofs in \cite{RM} using the strong double recursion theorem. In Section 3, we discuss some applications of the notion of effective inseparability in the meta-mathematics of arithmetic.

One main result in Section 2 is that we give a simpler and more efficient proof of the equivalences of effective inseparability using the strong double recursion theorem. Even if these equivalences have been proved by Smullyan, the proof is very complex for us and only one direction uses the double recursion theory. We directly prove some directions which are indirect in Smullyan's proof using a new method (strong double recursion theory). Our proof diagram in Theorem \ref{big thm} is much simper than Smullyan's proof diagram in Theorem \ref{equiv of EI}.
Section 3 is partly (Section 3.1) an exposition  of some of the literature but it also  contains original results in Section 3.2-3.4.

\section{Alternative characterizations of effectively inseparable pairs}

Smullyan proved in \cite{RM} that many different notions (see Definition \ref{def of EI}) of pair of RE sets are equivalent to the notion of effectively inseparable pair of RE sets. Smullyan's results are summarized in Theorem \ref{equiv of EI} which is not specifically stated in \cite{RM} as a theorem.
In this section, we give a much simper and more efficient proof of Theorem \ref{equiv of EI} to establish the equivalence of notions in Definition \ref{def of EI}. The main tool we use is the strong double recursion theorem for which we refer to \cite{RM}.

\subsection{Basic definitions and facts}

We first introduce some definitions. Our notations are standard.
\begin{definition}[Basic notations]~\label{}
\begin{enumerate}[(1)]
  \item For any set $A$ and function $f(x)$, define $f^{-1}[A]=\{x: f(x)\in A\}$.
  \item We denote the \emph{recursively enumerable} set with index $i$  by $W_i$. I.e., $W_i=\{x: \exists y \, T_1(i,x, y)\}$, where $T_1(z,x, y)$ is the Kleene predicate (cf.\cite{Kleene}). We denote the recursive function with index $e$  by $\phi_e$.
\end{enumerate}
\end{definition}

We list some examples of  $\sf EI$ pairs of RE sets.
\begin{example}~\label{}
\begin{itemize}
  \item $(A,B)$ is $\sf EI$, where $n\in A\Leftrightarrow \exists p [T_1((n)_0, n,p)\wedge \forall q\leq p\neg T_1((n)_1, n,q)]$ and $n\in B\Leftrightarrow \exists q [T_1((n)_1, n,q)\wedge \forall p\leq q\neg T_1((n)_0, n,p)]$\footnote{If $n=\langle a,b\rangle$, then $(n)_0=a$ and $(n)_1=b$.} (see \cite{Kleene}).
      \item $(A_i, A_j)$ is $\sf EI$ for $i\neq j$, where $A_i=\{e: \phi_e(e)=i\}$ (see \cite{Rogers}).
  \item If $(A, B)$ is a disjoint pair of non-empty RE sets, then $(X,Y)$ is $\sf EI$ where $X=\{e: \phi_e(e)\in A\}$ and $Y=\{e: \phi_e(e)\in B\}$.
\end{itemize}
\end{example}

\begin{remark}
We give some examples of $\sf RI$ pairs of RE sets which are not $\sf EI$. Ershov \cite{Ershov} shows that there is a disjoint pair $(A,B)$ of RE sets such that both $A$ and $B$ are creative,  $(A,B)$ is $\sf RI$ but is not $\sf EI$.
As a corollary of a theorem by Friedberg and Yates,  there exist recursive functions $\sigma_1$ and $\sigma_2$ such that if $W_e$ is non-recursive, then $(W_{\sigma_1(e)}, W_{\sigma_2(e)})$\footnote{In fact, $W_{\sigma_1(e)}$ and  $W_{\sigma_2(e)}$ are Turing incomparable (see \cite{Cleave}).} is $\sf RI$ but is not $\sf EI$ (see \cite{Cleave}).
\end{remark}

Notions in Definition \ref{def of EI} are introduced in \cite{RM}.
Smullyan proved in \cite{RM} that these notions  are equivalent to the notion of effectively inseparable pair of RE sets. We will give a much simper and more efficient proof of these equivalences (see Theorem \ref{big thm}).

\begin{definition}[\cite{RM}]\label{def of EI}
Let $(A,B)$ be a disjoint pair of RE sets.
\begin{enumerate}[(1)]
  \item
We say $(A,B)$ has a \emph{separation function} (denoted by {\sf SF}) if there is a recursive function $S(x,y,z)$ such that for any RE relations $M_1(x,y)$ and $M_2(x,y)$, there is $h\in\omega$ such that for any $x, y\in\omega$, we have:
\begin{enumerate}[(i)]
  \item $M_1(x,y)\wedge\neg M_2(x,y)\Rightarrow S(h,x,y)\in A$;
  \item $M_2(x,y)\wedge\neg M_1(x,y)\Rightarrow S(h,x,y)\in B$.
\end{enumerate}
  \item
We call a recursive function $f(x,y)$ a \emph{Kleene function}  for $(A,B)$ if for any $x$ and $y$, we have:
\begin{enumerate}[(i)]
  \item if $f(x,y)\in W_y- W_x$, then $f(x,y)\in A$;
  \item if $f(x,y)\in W_x- W_y$, then $f(x,y)\in B$.
\end{enumerate}
We say $(A,B)$ is a \emph{Kleene pair} (denoted by {\sf KP}) if it has a Kleene function.
  \item
We say $(A,B)$ is \emph{weakly doubly co-productive} $({\sf WDCP})$ if there is a recursive function $f(x,y)$ such that for any $i,j\in\omega$, we have:

\begin{enumerate}[(i)]
  \item if $W_i=W_j=\emptyset$, then $f(i,j)\notin A\cup B$;
  \item if $W_i=\emptyset$ and $W_j=\{f(i,j)\}$, then $f(i,j)\in B$;
  \item   $W_i=\{f(i,j)\}$ and $W_j=\emptyset$, then $f(i,j)\in A$.
\end{enumerate}
\item
We say $(A,B)$ is \emph{doubly co-productive} $({\sf DCP})$ if there is a recursive function $f(x,y)$ such that for any $i,j\in\omega$, if $W_i\cap W_j=\emptyset$ and $W_i\cap A=\emptyset$ and $W_j\cap B=\emptyset$, then $f(i,j)\notin A\cup B\cup W_i\cup W_j$.
\item We say $(A,B)$ is \emph{semi-${\sf DG}$} if there is a recursive function $f(x,y)$ such that for any $i,j\in\omega$, if $W_i\cap W_j=\emptyset$, then we have:
\begin{enumerate}[(i)]
  \item if $f(i,j)\in W_i$, then $f(i,j)\in A$;
  \item if $f(i,j)\in W_j$, then $f(i,j)\in B$.
\end{enumerate}
  \item
We say $(A,B)$ is \emph{doubly generative} $({\sf DG})$ if there is a recursive function $f(x,y)$ such that for any $i,j\in\omega$, if $W_i\cap W_j=\emptyset$, then we have:
\begin{enumerate}[(i)]
  \item $f(i,j)\in A$ iff $f(i,j)\in W_i$;
  \item $f(i,j)\in B$ iff $f(i,j)\in W_j$.
\end{enumerate}
\item We say $(A,B)$
is \emph{semi-reducible} to $(C,D)$ if there is a recursive function $f(x)$ such that if $x\in A$, then $f(x)\in C$, and if $x\in B$, then $f(x)\in D$.
We say $(C,D)$ is \emph{semi-doubly universal} (semi-${\sf DU})$ if  any disjoint pair $(A,B)$ of RE sets is semi-reducible to $(C,D)$.

\item
We say $(A,B)$ is \emph{reducible} to $(C,D)$ if there is a recursive function $f(x)$ such that $x\in A$ iff $f(x)\in C$ and $x\in B$ iff $f(x)\in D$.
We say $(C,D)$ is \emph{doubly universal} $({\sf DU})$ if  any disjoint pair $(A,B)$ of RE sets is reducible to $(C,D)$.
\item
We say $(A,B)$ is \emph{weakly effective inseparable} $({\sf WEI})$ if there is a recursive function $f(x,y)$ such that for any $i$ and $j$, we have:
\begin{enumerate}[(i)]
  \item if $W_i=A$ and $W_j=B$, then $f(i,j)\notin A\cup B$;
  \item if $W_i=A$ and $W_j=B\cup\{f(i,j)\}$, then $f(i,j)\in A$;
  \item if $W_i=A\cup\{f(i,j)\}$ and $W_j=B$, then $f(i,j)\in B$.
\end{enumerate}
\item We say $(A,B)$ is \emph{completely effective inseparable} $({\sf CEI})$ if there is a recursive function $f(x,y)$ such that for any $i$ and $j$, if $A\subseteq W_i$ and $B\subseteq W_j$, then $f(i,j)\in W_i\Leftrightarrow f(i,j)\in W_j$.
\end{enumerate}
\end{definition}

\begin{fact}[\cite{Myhill, Smullyan61}]\label{fact about DU}
The following are equivalent:
\begin{itemize}
  \item A disjoint pair $(A,B)$ of RE sets is ${\sf DU}$;
  \item Every disjoint pair of RE sets is reducible to $(A,B)$ under a 1-1 function $g(x)$.
\end{itemize}
\end{fact}

\subsection{The reduction theorem}

In this section, we will show in Theorem \ref{reduction thm} that if $(A,B)$ has any property $P$ in Definition \ref{def of EI} and $(A,B)$ is reducible to $(C,D)$, then $(C,D)$ also has the property $P$. We call this fact the reduction theorem.  However, it is not  true that if  $(C,D)$ has any property $P$ in Definition \ref{def of EI} and $(A,B)$ is reducible to $(C,D)$, then $(A,B)$ has the property $P$.

\begin{proposition}[Smullyan]~\label{reduction prop}
\begin{enumerate}[(1)]
  \item If $(A,B)$ is a Kleene pair and $(A,B)$ is semi-reducible to $(C,D)$, then $(C,D)$ is a Kleene pair \textup{(}Lemma $A_1$, p.70, \cite{RM}\textup{)}.
  \item If $(A,B)$ is $\sf EI$ and $(A,B)$ is semi-reducible to $(C,D)$, then $(C,D)$ is $\sf EI$ \textup{(}Proposition 1, p.220, \cite{Smullyan}\textup{)}.
\end{enumerate}
\end{proposition}

The main tool we use in Theorem \ref{reduction thm} is the s-m-n theorem.

\begin{theorem}[The s-m-n theorem, \cite{Rogers}]\label{}
For any $m, n\geq 1$, there exists a recursive function $s_n^m$ of $m+1$ variables such that for all $x, y_1, \cdots, y_m, z_1, \cdots, z_n$, we have
\[\phi^n_{s_n^m(x, y_1, \cdots, y_m)}(z_1, \cdots, z_n)=\phi_{x}^{m+n}(y_1, \cdots, y_m, z_1, \cdots, z_n).\]
\end{theorem}

Now, we prove the reduction theorem.

\begin{theorem}[The reduction theorem]\label{reduction thm}
Let $(A,B)$ and $(C,D)$ be disjoint pairs of RE sets.
\begin{enumerate}[(1)]
  \item  If $(A,B)$ is semi-$\sf DU$ and $(A,B)$ is semi-reducible to $(C,D)$, then $(C,D)$ is semi-$\sf DU$.
  \item If $(A,B)$ is $\sf DU$ and $(A,B)$ is reducible to $(C,D)$, then $(C,D)$ is $\sf DU$.
      \item If $(A,B)$ has a separation function and $(A,B)$ is reducible to $(C,D)$, then $(C,D)$ has a separation function.
  \item If $(A,B)$ is $\sf CEI$  and $(A,B)$ is reducible to $(C,D)$, then $(C,D)$ is $\sf CEI$.
  \item If $(A,B)$ is $\sf DG$  and $(A,B)$ is reducible to $(C,D)$, then $(C,D)$ is $\sf DG$.
  \item If $(A,B)$ is semi-$\sf DG$  and $(A,B)$ is reducible to $(C,D)$, then $(C,D)$ is semi-$\sf DG$.
  \item If $(A,B)$ is $\sf DCP$  and $(A,B)$ is reducible to $(C,D)$, then $(C,D)$ is $\sf DCP$.
      \item
If $(A,B)$ is $\sf WEI$ and $(A,B)$ is reducible to $(C,D)$, then $(C,D)$ is $\sf WEI$.
      \item  If $(A,B)$ is $\sf WDCP$ and $(A,B)$ is reducible to $(C,D)$, then $(C,D)$ is $\sf WDCP$.
\end{enumerate}
\end{theorem}
\begin{proof}\label{}
Items (1) and (2) are easy to check.

(3): Suppose $(A,B)$ has a separation function $f(x,y,z)$ and $(A,B)$ is reducible to $(C,D)$ via the function $g(x)$. It is easy to check that $s(x,y,z)=g(f(x,y,z))$ is a separation function for $(C,D)$.

We show that (4)-(9) hold. Let $P$ be any one of the following properties: $\sf CEI, \sf DG$, semi-$\sf DG$, $\sf DCP$, $\sf WEI$, $\sf WDCP$.
Suppose $(A,B)$ has the property $P$ via the recursive function $f(x,y)$ and $(A,B)$ is reducible to $(C,D)$ via the recursive function $g(x)$.
By the s-m-n theorem, there is a recursive function $h(x)$ such that $g^{-1}[W_i]=W_{h(i)}$ for any $i$.
Define $s(i,j)=g(f(h(i), h(j)))$. Note that the function $s$ is recursive. We show that $(C,D)$ has the property $P$ via the recursive function $s$.

(4): Let $P$ be $\sf CEI$. We show that $(C,D)$ is $\sf CEI$ via the recursive function $s$: if $C\subseteq W_i$ and $D\subseteq W_j$, then $s(i,j)\in W_i$ if and only if $s(i,j)\in W_j$.

Suppose $C\subseteq W_i$ and $D\subseteq W_j$. Note that $A= g^{-1}[C]\subseteq g^{-1}[W_i]=W_{h(i)}$ and $B= g^{-1}[D]\subseteq g^{-1}[W_j]=W_{h(j)}$. Since $(A,B)$ is $\sf CEI$ via the recursive function $f$, we have: $f(h(i), h(j))\in W_{h(i)} \Leftrightarrow f(h(i), h(j))\in W_{h(j)}$. Since $g^{-1}[W_i]=W_{h(i)}, g^{-1}[W_j]=W_{h(j)}$ and $s(i,j)=g(f(h(i), h(j)))$, we have $s(i,j)\in W_i\Leftrightarrow s(i,j)\in W_j$.

(5): Let $P$ be $\sf DG$. We show that $(C,D)$ is $\sf DG$ via the recursive function $s$:  if $W_i\cap W_j=\emptyset$, then $s(i,j)\in C\Leftrightarrow s(i,j)\in W_i$ and $s(i,j)\in D\Leftrightarrow s(i,j)\in W_j$.

Suppose $W_i\cap W_j=\emptyset$. Since $W_{h(i)}\cap W_{h(j)}=\emptyset$ and $(A,B)$ is $\sf DG$ via the recursive function $f$, we have $f(h(i), h(j))\in A\Leftrightarrow f(h(i), h(j))\in W_{h(i)}$.  Since $A= g^{-1}[C]$, we have
\[s(i,j)\in C \Leftrightarrow  f(h(i), h(j))\in A\Leftrightarrow f(h(i), h(j))\in W_{h(i)}=g^{-1}[W_i] \Leftrightarrow s(i,j)\in W_i. \]
By a symmetric argument, we have $s(i,j)\in D\Leftrightarrow s(i,j)\in W_j$.

(6): The argument is similar to (5).

(7): Let $P$ be $\sf DCP$. We show that $(C,D)$ is $\sf DCP$ via the recursive function $s$:  if $W_i\cap W_j=\emptyset, W_i\cap C=\emptyset$ and $W_j\cap D=\emptyset$, then $s(i,j)\notin C\cup D\cup W_i\cup W_j$.

Note that
$W_{h(i)}\cap W_{h(j)}=\emptyset, g^{-1}[W_i]\cap g^{-1}[C]=\emptyset, g^{-1}[W_j]\cap g^{-1}[D]=\emptyset, A=g^{-1}[C]$ and $B=g^{-1}[D]$. Since $(A,B)$ is $\sf DCP$ via the recursive function $f$, $W_{h(i)}\cap A=\emptyset$ and $W_{h(j)}\cap B=\emptyset$, we have $f(h(i), h(j))\notin A\cup B\cup W_{h(i)}\cup W_{h(j)}$. Thus, $s(i,j)\notin C\cup D\cup W_i\cup W_j$.

(8): Let $P$ be $\sf WEI$. We show that $(C,D)$ is $\sf WEI$ via the recursive function $s$: (I) if $W_i=C$ and $W_j=D$, then $s(i,j)\notin C\cup D$; (II) if $W_i=C$ and $W_j=D\cup \{s(i,j)\}$, then $s(i,j)\in C$; (III) if $W_i=C\cup \{s(i,j)\}$ and $W_j=D$, then $s(i,j)\in D$.

\begin{enumerate}[(I)]
  \item Suppose $W_i=C$ and $W_j=D$. Note that $A=g^{-1}[C]=g^{-1}[W_i]=W_{h(i)}$, $B=g^{-1}[D]=g^{-1}[W_j]=W_{h(j)}$.
Since $(A,B)$ is $\sf WEI$ via the recursive function $f$, $f(h(i), h(j))\notin A\cup B$. Thus, $s(i,j)\notin C\cup D$.
  \item Suppose $W_i=C$ and $W_j=D\cup \{s(i,j)\}$.  Note that $A=g^{-1}[W_i]=W_{h(i)}$, $B\cup g^{-1}[\{s(i,j)\}]=g^{-1}[W_j]=W_{h(j)}$.
      From \cite{RM}, $\sf WEI$ implies $\sf DU$. From (2), $(C,D)$ is $\sf DU$. By Fact \ref{fact about DU}, we can assume that  $g$ is injective, and thus $g^{-1}[\{s(i,j)\}]=\{f(h(i), h(j))\}$. Since $(A,B)$ is $\sf WEI$ via $f$, $f(h(i), h(j))\in A$. Thus, $s(i,j)\in C$.
  \item Follows by a symmetric argument as for (II).
\end{enumerate}

(9): Let $P$ be $\sf WDCP$. We show that $(C,D)$ is is $\sf WDCP$ via the recursive function $s$: (I) if $W_i=W_j=\emptyset$, then $s(i,j)\notin C\cup D$; (II) if $W_i=\emptyset$ and $W_j=\{s(i,j)\}$,  then $f(i,j)\in D$; (III) if $W_i=\{s(i,j)\}$ and $W_j=\emptyset$,  then $f(i,j)\in C$.

\begin{enumerate}[(I)]
  \item Suppose $W_i=W_j=\emptyset$. Since $W_{h(i)}=W_{h(j)}=\emptyset$ and $(A,B)$ is  $\sf WDCP$ via the recursive function $f$, $f(h(i), h(j))\notin A\cup B$. Thus, $s(i,j)\notin C\cup D$.
  \item Suppose $W_i=\emptyset$ and $W_j=\{s(i,j)\}$.  From \cite{RM}, $\sf WDCP$ implies $\sf DU$. From (2), $(C,D)$ is $\sf DU$. By Fact \ref{fact about DU}, we can assume that  $g$ is injective. Then $W_{h(i)}=\emptyset$ and $W_{h(j)}=g^{-1}[W_j]=g^{-1}[\{s(i,j)\}]=\{f(h(i), h(j))\}$. Since $(A,B)$ is $\sf WDCP$ via the recursive function $f$, $f(h(i), h(j))\in B$. Thus, $s(i,j)\in D$.
  \item Follows by a symmetric argument as for (II).
\end{enumerate}

\end{proof}

It is easy to see that for a disjoint pair $(A,B)$ of RE sets, $(A,B)$ is $\sf DU$ (semi-$\sf DU$, has separation function) if and only if $(B,A)$ is $\sf DU$ (semi-$\sf DU$, has separation function). The following proposition is an easy observation.

\begin{proposition}
For any disjoint pair $(A,B)$ of RE sets, if $(A,B)$ is $\sf CEI\, (\sf EI, \sf WEI, \sf DG$, Semi-$\sf {DG}, {\sf DCP}$, ${\sf WDCP}$, Kleene pair) via the recursive function $f(i,j)$, then $(B,A)$ is $\sf CEI (\sf EI, \sf WEI, \sf DG$, Semi-$\sf {DG}, {\sf DCP}$, ${\sf WDCP}$, Kleene pair) via the recursive function $g(i,j)=f(j,i)$.
\end{proposition}

\begin{corollary}
Let $P$ be any property  in Definition \ref{def of EI} and $(A,B)$ be any disjoint pair of RE sets. Then $(A,B)$ has the property $P$ iff $(B,A)$ has the property $P$.
\end{corollary}

\subsection{A simper proof of Smullyan's theorem}

Theorem \ref{equiv of EI} is not specifically stated as a theorem in \cite{RM}, but results in \cite{RM} in fact imply Theorem \ref{equiv of EI} from which notions in Definition \ref{def of EI} are equivalent to the notion of effective inseparability.
In this section, we give a much simper proof of Theorem \ref{equiv of EI} via the strong double recursion theorem.

\begin{theorem}[Smullyan, \cite{RM}]\label{equiv of EI}
Let $(A,B)$ be a disjoint pair of RE sets, and $P$ be any one of the following properties: Kleene pair (${\sf KP}$), having a separation function (${\sf SF}$), ${\sf WEI}$, ${\sf CEI}$, semi-${\sf DG}$, ${\sf DG}$, semi-${\sf DU}$, ${\sf DU}$, ${\sf WDCP}$ and ${\sf DCP}$. Then  $(A,B)$  has the property $P$ if and only if $(A,B)$ is ${\sf EI}$.
\end{theorem}

\begin{proof}\label{}
Smullyan proved the following diagram in \cite{RM}. All notions in this diagram\footnote{I would like to thank my student for the help to draw the pictures.}  are equivalent.

\begin{center}

\xymatrix@C=2ex{
&&&\sf SF\ar@{=>}"2,3"\\
&&\sf DU\ar@{=>}"2,5"&&\sf semi\mbox{-}DU\ar@{=>}"1,4"\ar@{=>}"2,7"&&\sf KP\ar@{=>}"4,7"\\
\sf DCP\ar@{=>}"4,1"&&\sf DG\ar@{=>}"3,1"\ar@{=>}"2,3"\ar@{=>}"3,5"&&\sf semi\mbox{-}DG\ar@{=>}"2,7"\\
\sf WDCP\ar@{=>}"3,3"&&\sf WEI\ar@{=>}"4,1"&&\sf EI\ar@{=>}"4,3"&&\sf CEI\ar@{=>}"4,5"}

\end{center}

\begin{enumerate}
  \item {\sf SF} $\Rightarrow {\sf DU}$: Theorem 2 in p.91, \cite{RM}.
  \item ${\sf DU}\Rightarrow$ semi-${\sf DU}$: by definition.
  \item semi-${\sf DU}\Rightarrow {\sf SF}$: Proposition 1 in p.90, \cite{RM}.
  \item ${\sf DG}\Rightarrow {\sf DU}$: Theorem 2 in p.86, \cite{RM}.
      \item semi-${\sf DU}\Rightarrow {\sf KP}$: Theorem 9 in p.75, \cite{RM}.
          \item ${\sf DG}\Rightarrow$ semi-${\sf DG}$: by definition.
              \item semi-${\sf DG}\Rightarrow {\sf KP}$: Theorem 13 in p.79, \cite{RM}.
                  \item ${\sf KP}\Rightarrow {\sf CEI}$: Proposition 2 in p.68, \cite{RM}.
                      \item ${\sf CEI}\Rightarrow {\sf EI}$: by definition.
\item ${\sf EI}\Rightarrow {\sf WEI}$: by definition.
\item ${\sf WEI}\Rightarrow {\sf WDCP}$: Theorem 3 in p.126, \cite{RM}.
    \item ${\sf WDCP}\Rightarrow {\sf DG}$: Theorem 2 in p.123, \cite{RM}.
        \item ${\sf DG}\Rightarrow {\sf DCP}$: by definition (easy to check).
            \item ${\sf DCP}\Rightarrow {\sf WDCP}$: by definition.
\end{enumerate}
\end{proof}

For the above proofs in \cite{RM}, only the proof of ``${\sf WDCP}\Rightarrow {\sf DG}$" uses the double recursion theorem. In the following, we will give a much simpler and more direct proof of Theorem \ref{equiv of EI}   using the $\mathbf{strong}$ double recursion theorem (see Theorem \ref{big thm}). Our main observation is: using the $\mathbf{strong}$ double recursion theorem, we can directly prove many implication relations among the notions  in Theorem \ref{equiv of EI}. We first introduce the strong double recursion theorem in \cite{RM}.

\begin{theorem}[The strong double recursion theorem, Theorem 2 in p.107, \cite{RM}]\label{SDRT}
For any RE relations $R_1(x,y_1, y_2, z_1,z_2)$ and $R_2(x,y_1, y_2, z_1,z_2)$, there are recursive functions $t_1(y_1, y_2)$ and $t_2(y_1, y_2)$ such that for all $i$ and $j$, we have:
\begin{enumerate}[(1)]
  \item $W_{t_1(i,j)}=\{x: R_1(x, i, j, t_1(i,j), t_2(i,j))\}$;
  \item $W_{t_2(i,j)}=\{x: R_2(x, i, j, t_1(i,j), t_2(i,j))\}$.
\end{enumerate}
\end{theorem}

\begin{lemma}\label{coro of DRT}
For any RE relation $M_1(x,y,z_1, z_2)$ and $M_2(x,y,z_1, z_2)$, there are recursive functions $t_1(y_1,y_2)$ and $t_2(y_1,y_2)$ such that
\begin{enumerate}[(1)]
\item $W_{t_1(y_1,y_2)}=\{x: M_1(x,y_2, t_1(y_1,y_2), t_2(y_1,y_2))\}$;
\item $W_{t_2(y_1,y_2)}=\{x: M_2(x,y_1, t_1(y_1,y_2), t_2(y_1,y_2))\}$.
\end{enumerate}
\end{lemma}
\begin{proof}\label{}
Apply Theorem \ref{SDRT} to $R_1(x,y_1, y_2, z_1,z_2)= M_1(x,y_2,z_1, z_2)$ and $R_2(x, y_1, y_2, z_1,z_2)= M_2(x,y_1,z_1, z_2)$.
\end{proof}

\begin{remark}
Lemma \ref{coro of DRT} has different variants. For instance, we also have that for any RE relation $M_1(x,y,z_1, z_2)$ and $M_2(x,y,z_1, z_2)$, there are recursive functions $t_1(y_1,y_2)$ and $t_2(y_1,y_2)$ such that $W_{t_1(y_1,y_2)}=\{x: M_1(x,y_1, t_1(y_1,y_2), t_2(y_1,y_2))\}$ and $W_{t_2(y_1,y_2)}=\{x: M_2(x,y_2, t_1(y_1,y_2), t_2(y_1,y_2))\}$. But in the proof of Theorem \ref{WEI IP DG}, what we need is the form as in Lemma \ref{coro of DRT}.
\end{remark}

Now, we directly prove that ${\sf WEI}$ implies ${\sf DG}$. The main tool we use is Lemma \ref{coro of DRT} which follows from the strong double recursion theorem.

\begin{theorem}\label{WEI IP DG}
If $(A,B)$ is  ${\sf WEI}$, then $(A,B)$ is ${\sf DG}$.
\end{theorem}
\begin{proof}\label{}
Suppose $(A,B)$ is  ${\sf WEI}$ via the recursive function $g$:
\begin{enumerate}[(I)]
  \item if $W_i=A$ and $W_j=B$, then $g(i,j)\notin A\cup B$;
  \item if $W_i=A$ and $W_j=B\cup\{g(i,j)\}$, then $g(i,j)\in A$;
  \item  if $W_i=A\cup\{g(i,j)\}$ and $W_j=B$, then $g(i,j)\in B$.
\end{enumerate}

Define:
\[M_1(x,y,z_1, z_2): x\in A \vee (x=g(z_1, z_2)\wedge x\in W_y);\]
\[M_2(x,y,z_1, z_2): x\in B \vee (x=g(z_1, z_2)\wedge x\in W_y).\]
Apply Lemma \ref{coro of DRT} to RE relations $M_1(x,y,z_1, z_2)$ and $M_2(x,y,z_1, z_2)$. Then there are recursive functions $t_1(y_1,y_2)$ and $t_2(y_1,y_2)$ such that
\begin{enumerate}[(1)]
\item $W_{t_1(i,j)}=\{x: M_1(x, j, t_1(i,j), t_2(i,j))\}=A\cup (W_j\cap \{g(t_1(i,j), t_2(i,j))\})$;
\item $W_{t_2(i,j)}=\{x: M_2(x, i, t_1(i,j), t_2(i,j))\}=B\cup (W_i\cap \{g(t_1(i,j), t_2(i,j))\})$.
\end{enumerate}
Define $f(i,j)=g(t_1(i,j), t_2(i,j))$. Clearly, $f$ is recursive. Note that \[W_{t_1(i,j)}=A\cup (W_j\cap \{f(i,j)\})\]
 and \[W_{t_2(i,j)}=B\cup (W_i\cap \{f(i,j)\}).\]

We show that $(A,B)$ is ${\sf DG}$ via $f$: if $W_i\cap W_j=\emptyset$, then $f(i,j)\in W_i\Leftrightarrow f(i,j)\in A$ and $f(i,j)\in W_j\Leftrightarrow f(i,j)\in B$.

Suppose $W_i\cap W_j=\emptyset$.
\begin{enumerate}[(1)]
  \item If $f(i,j)\in W_i$, then $W_{t_1(i,j)}=A$ and $W_{t_2(i,j)}=B\cup  \{f(i,j)\}$. By condition (II), we have $f(i,j)\in A$.
  \item If $f(i,j)\in W_j$, then $W_{t_1(i,j)}=A\cup  \{f(i,j)\}$ and $W_{t_2(i,j)}=B$. By condition (III), we have $f(i,j)\in B$.
  \item If $f(i,j)\notin W_i\cup W_j$, then $W_{t_1(i,j)}=A$ and $W_{t_2(i,j)}=B$. By condition (I), we have $f(i,j)\notin A\cup B$.
\end{enumerate}
Thus, $f(i,j)\in W_i\Leftrightarrow f(i,j)\in A$ and $f(i,j)\in W_j\Leftrightarrow f(i,j)\in B$.
\end{proof}

Theorem \ref{thm 2.4} is a corollary of Theorem \ref{SDRT}, the strong double recursion theorem (see pp.107-108 in \cite{RM}).

\begin{theorem}[Theorem 2.4 in p.108, \cite{RM}]\label{thm 2.4}
For any two RE relations $M_1(x,y,z)$ and $M_2(x,y,z)$ and any recursive function $g(x,y)$, there are recursive functions $f_1(y)$ and $f_2(y)$ such that for any $y$,
\begin{enumerate}[(1)]
  \item $W_{f_1(y)}=\{x: M_1(x,y, g(f_1(y), f_2(y)))\}$;
  \item $W_{f_2(y)}=\{x: M_2(x,y, g(f_1(y), f_2(y)))\}$.
\end{enumerate}
\end{theorem}

Now, we directly prove that ${\sf WDCP}$ implies ${\sf DU}$. The main tool we use is
Theorem \ref{thm 2.4}, a corollary of the strong double recursion theorem.

\begin{theorem}\label{WDCP IP DU}
If $(C,D)$ is  ${\sf WDCP}$, then $(C,D)$ is ${\sf DU}$.
\end{theorem}
\begin{proof}\label{}
Suppose $(C,D)$ is  ${\sf WDCP}$ via the recursive function $g$:
\begin{enumerate}[(I)]
  \item if $W_i=\{g(i,j)\}$ and $W_j=\emptyset$, then $g(i,j)\in C$;
  \item if $W_i=\emptyset$ and $W_j=\{g(i,j)\}$, then $g(i,j)\in D$;
  \item  If $W_i=W_j=\emptyset$, then $f(i,j)\notin C\cup D$.
\end{enumerate}

We show that $(C,D)$ is ${\sf DU}$: any disjoint pair of RE sets is reducible to $(C,D)$.  Let $(A,B)$ be a disjoint pair of RE sets.
Define \[M_1(x,y,z): y\in A\wedge x=z\]
 and  \[M_2(x,y,z): y\in B\wedge x=z.\]
Apply Theorem \ref{thm 2.4} to  $M_1(x,y,z)$ and $M_2(x,y,z)$. Then there are recursive functions $f_1(y)$ and $f_2(y)$ such that for any $y$:
\[W_{f_1(y)}=\{x: y\in A\wedge x=g(f_1(y), f_2(y))\}\]
 and  \[W_{f_2(y)}=\{x: y\in B\wedge x=g(f_1(y), f_2(y))\}.\]
Define $h(y)=g(f_1(y), f_2(y))$.

We show that $(A,B)$ is reducible to $(C,D)$ via the recursive function $h$.
\begin{enumerate}[(1)]
  \item Suppose $y\in A$. Then $W_{f_1(y)}=\{g(f_1(y), f_2(y))\}$ and $W_{f_2(y)}=\emptyset$. Thus, by condition (I), $h(y)\in C$.
  \item Suppose $y\in B$. Then $W_{f_1(y)}=\emptyset$ and $W_{f_2(y)}=\{g(f_1(y), f_2(y))\}$. Thus, by condition (II), $h(y)\in D$.
  \item Suppose $y\notin A\cup B$. Then $W_{f_1(y)}=W_{f_2(y)} =\emptyset$. Thus, by condition (III), $h(y)\notin C\cup D$.
\end{enumerate}
Thus, $y\in A\Leftrightarrow h(y)\in C$ and $y\in B\Leftrightarrow h(y)\in D$.
\end{proof}

Now, we give a much simper and more efficient proof of Theorem \ref{equiv of EI} via
the strong double recursion theorem.

\begin{theorem}\label{big thm}
Theorem \ref{equiv of EI} can be proved via the strong double recursion theorem as in the following picture:

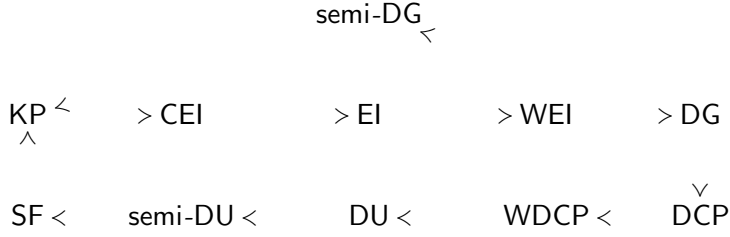
\begin{center}

\xymatrix{
&&\sf semi\mbox{-}DG\ar@{=>}"2,1"\\
\sf KP\ar@{=>}"2,2"&\sf CEI\ar@{=>}"2,3"&\sf EI\ar@{=>}"2,4"&\sf WEI\ar@{=>}"2,5"&\sf DG\ar@{=>}"1,3"\ar@{=>}"3,5"\\
\sf SF\ar@{=>}"2,1"&\sf semi\mbox{-}DU\ar@{=>}"3,1"&\sf DU\ar@{=>}"3,2"&\sf WDCP\ar@{=>}"3,3"&\sf DCP\ar@{=>}"3,4"}

\end{center}

\end{theorem}
\begin{proof}\label{}
\begin{enumerate}[(1)]
  \item ${\sf KP}\Rightarrow {\sf CEI}$: Proposition 2 in p.68, \cite{RM}.
  \item ${\sf CEI} \Rightarrow {\sf EI}$: by definition.
  \item ${\sf EI}\Rightarrow {\sf WEI}$: by definition.
  \item ${\sf WEI}\Rightarrow {\sf DG}$: Theorem \ref{WEI IP DG}.
  \item ${\sf DG}\Rightarrow {\sf DCP}$: by definition (easy to check).
  \item ${\sf DCP}\Rightarrow {\sf WDCP}$: by definition.
  \item  ${\sf WDCP}\Rightarrow {\sf DU}$: Theorem \ref{WDCP IP DU}.
  \item ${\sf DU}\Rightarrow$ semi-${\sf DU}$: by definition.
  \item semi-${\sf DU}\Rightarrow {\sf SF}$: Proposition 1 in p.90, \cite{RM}.
  \item ${\sf SF}\Rightarrow {\sf KP}$:
Suppose $(A,B)$ has a separation function. By Theorem 2 of \cite{RM} in p.91, $(A,B)$ is ${\sf DU}$ (i.e., any disjoint pair of RE sets is reducible to $(A,B)$). By Theorem 1 of \cite{RM} in p.68, there exists a Kleene pair $(K_1, K_2)$. Since $(K_1, K_2)$ is reducible to $(A,B)$, by Proposition \ref{reduction prop}, $(A,B)$ is a Kleene pair.
  \item ${\sf DG}\Rightarrow$ semi-${\sf DG}$: by definition.
  \item semi-${\sf DG}\Rightarrow {\sf KP}$: Theorem 13 in p.79, \cite{RM}.
\end{enumerate}

\end{proof}

Our proof of Theorem \ref{big thm} via the strong double recursion theorem is much
simpler than the corresponding proof in \cite{RM}.
Among notions in Definition \ref{def of EI}, five key notions are ${\sf DG}, {\sf WEI}, {\sf DU}, {\sf CEI}, {\sf EI}$. Smullyan  has shown in \cite{RM} that these notions are equivalent in a complex way. In the rest of this section, we prove the equivalence of these notions via a much simpler way (Theorem \ref{simple proof}) than those proofs in \cite{RM} using the strong double recursion theorem.

\begin{theorem}[Theorem 4, p.57, \cite{RM}]\label{special fun}
There is a recursive function $\sigma(x,y)$ such that for all $i$ and $j$,
\begin{enumerate}[(1)]
  \item $W_{\sigma(i,j)}$ and $W_{\sigma(j,i)}$ are disjoint supersets of $W_i- W_j$ and $W_j- W_i$ respectively;
  \item If $W_i$ and $W_j$ are disjoint, then $W_i=W_{\sigma(i,j)}$ and $W_j=W_{\sigma(j,i)}$.
\end{enumerate}
\end{theorem}

We directly prove that ${\sf DG}$ implies ${\sf WEI}$. The main tool we use is Theorem \ref{special fun}.

\begin{proposition}\label{DG IMP WEI}
If $(C,D)$ is  ${\sf DG}$, then it is ${\sf WEI}$.
\end{proposition}
\begin{proof}\label{}
The idea of this proof is essentially from \cite{RM}. Suppose $(C,D)$ is  ${\sf DG}$ via the recursive function $k(x,y)$. Define $g(i,j)=k(\sigma(j,i), \sigma(i,j))$. We show that $(C,D)$ is  ${\sf WEI}$ via $g$.
\begin{enumerate}[(1)]
  \item Suppose $W_i=C$ and $W_j=D$. We show that $g(i,j)\notin C\cup D$.

Since $W_i\cap W_j=\emptyset$, by Theorem \ref{special fun}, $W_i=W_{\sigma(i,j)}$ and $W_j=W_{\sigma(j,i)}$. Since $(C,D)$ is  ${\sf DG}$ via the recursive function $k$, we have $g(i,j)\in C\Leftrightarrow g(i,j)\in W_j$ and
$g(i,j)\in D\Leftrightarrow g(i,j)\in W_i$.
Thus, $g(i,j)\notin C\cup D$.
  \item Suppose $W_i=C$ and $W_j=D\cup \{g(i,j)\}$. We show that $g(i,j)\in C$.

Suppose $g(i,j)\notin C$. Then $W_i\cap W_j=\emptyset$. Then $g(i,j)\in C\Leftrightarrow g(i,j)\in W_j$. So $g(i,j)\notin W_j$, which leads to a contradiction.

  \item Suppose $W_i=C\cup \{g(i,j)\}$ and $W_j=D$. By a symmetric argument as (2), we can show that $g(i,j)\in D$.
\end{enumerate}

Thus, we have shown that $(C,D)$ is  ${\sf WEI}$ via the recursive function $g$.
\end{proof}

The main tool we use in the proof of Lemma \ref{coro2 of DRT} is Theorem \ref{thm 2.4}, a corollary of the strong double recursion theorem.

\begin{lemma}\label{coro2 of DRT}
For any disjoint pair of RE sets $A$ and $B$, for any RE sets $C$ and $D$, and any recursive function $g(x,y)$, there are recursive functions $f_1(y)$ and $f_2(y)$ such that for any $y$,
\begin{enumerate}[(I)]
\item if $y\in A$, then $W_{f_1(y)}=C$ and $W_{f_2(y)}=D\cup \{g(f_1(y), f_2(y))\}$;
\item  if $y\in B$, then $W_{f_1(y)}=C\cup \{g(f_1(y), f_2(y))\}$ and $W_{f_2(y)}=D$;
\item if $y\notin A\cup B$, then $W_{f_1(y)}=C$ and $W_{f_2(y)}=D$.
\end{enumerate}
\end{lemma}
\begin{proof}\label{}
Define \[M_1(x,y,z): x\in C\vee (y\in B\wedge x=z)\] and \[M_2(x,y,z): x\in D\vee (y\in A\wedge x=z).\]
Apply Theorem \ref{thm 2.4} to $M_1(x,y,z)$ and $M_2(x,y,z)$. Then there are recursive functions $f_1(y)$ and $f_2(y)$ such that for any $y$:
\begin{enumerate}[(1)]
  \item $W_{f_1(y)}=\{x: M_1(x,y, g(f_1(y), f_2(y)))\}=\{x: x\in C\vee (y\in B\wedge x=g(f_1(y), f_2(y)))$;
  \item $W_{f_2(y)}=\{x: M_2(x,y, g(f_1(y), f_2(y)))\}=\{x: x\in D\vee (y\in A\wedge x=g(f_1(y), f_2(y)))\}$.
\end{enumerate}
It is easy to check that (I)-(III) hold.
\end{proof}

It is an exercise in \cite{RM} (p. 126) that ${\sf WEI}$ implies ${\sf DU}$. We give proof details of it here. The main tool we use is Lemma \ref{coro2 of DRT} which follows from the strong double recursion theorem.

\begin{theorem}[Exercise 1 in p. 126, \cite{RM}]\label{WEI IP DU}
If $(C,D)$ is  ${\sf WEI}$, then it is ${\sf DU}$.
\end{theorem}
\begin{proof}\label{}
Suppose  $(C,D)$ is  ${\sf WEI}$ via the recursive function $g$. Let $(A,B)$ be a disjoint pair of RE sets. Apply Lemma \ref{coro2 of DRT} to $A,B,C,D$ and $g$. Then there are recursive functions $f_1$ and $f_2$ such that conditions (I)-(III) in Lemma \ref{coro2 of DRT} hold.

We show that $(A,B)$ is reducible to $(C,D)$ via $h(y)=g(f_1(y), f_2(y))$.

\begin{enumerate}[(1)]
  \item  We show that if $y\in A$, then $h(y)\in C$.
Suppose $y\in A$. Then, by (I) in Lemma \ref{coro2 of DRT}, $W_{f_1(y)}=C$ and $W_{f_2(y)}=D\cup \{g(f_1(y), f_2(y))\}$. Since $(C,D)$ is  ${\sf WEI}$ via $g$, we have $h(y)=g(f_1(y), f_2(y))\in C$.

  \item We show that if $y\in B$, then $h(y)\in D$.
Suppose $y\in B$. Then, by (II) in Lemma \ref{coro2 of DRT}, $W_{f_1(y)}=C\cup \{g(f_1(y), f_2(y))\}$ and $W_{f_2(y)}=D$. Since $(C,D)$ is  ${\sf WEI}$ via $g$, we have $h(y)\in D$.
  \item We show that if $y\notin A\cup B$, then $h(y)\notin C\cup D$.
Suppose $y\notin A\cup B$. Then, by (III) in Lemma \ref{coro2 of DRT}, $W_{f_1(y)}=C$ and $W_{f_2(y)}=D$. Since $(C,D)$ is  ${\sf WEI}$ via $g$, we have $h(y)\notin C\cup D$.
\end{enumerate}
\end{proof}

\begin{theorem}\label{simple proof}
The equivalence of $\sf DG, \sf WEI, \sf DU, \sf CEI$ and $\sf EI$ can be proved  as in the following picture:

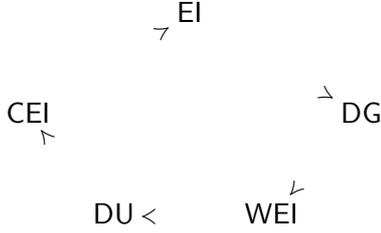
\begin{center}
\xymatrix@C=2ex{
&&\sf EI\ar@{=>}"2,5"\\
\sf CEI\ar@{=>}"1,3"&&&&\sf DG\ar@{=>}"3,4"\\
&\sf DU\ar@{=>}"2,1"&&\sf WEI\ar@{=>}"3,2"}
\end{center}

\end{theorem}
\begin{proof}\label{}
\begin{enumerate}[(1)]
  \item $\sf DG\Rightarrow \sf WEI$: Proposition \ref{DG IMP WEI}.
  \item $\sf WEI \Rightarrow \sf DU$: Theorem \ref{WEI IP DU}.
  \item $\sf DU\Rightarrow \sf CEI$: Since there exists a Kleene pair, by Proposition \ref{reduction prop}(1), if $(A,B)$ is $\sf DU$, then $(A,B)$ is a Kleene pair. By Proposition 2 of \cite{RM} in p.68, $(A,B)$ is $\sf CEI$.
   \item $\sf CEI \Rightarrow \sf EI$: By definition.
  \item $\sf EI\Rightarrow \sf DG$: Follows from Theorem \ref{WEI IP DG} since $\sf EI$ implies $\sf WEI$ by definition.
\end{enumerate}
\end{proof}

\begin{remark}
The proof of ``$\sf DG\Rightarrow \sf DU$" in \cite{RM} is complex and does not use any version of recursion theorems. From Theorem \ref{simple proof}, we also give a simpler proof of ``$\sf DG\Rightarrow \sf DU$" via the strong double recursion theorem.
\end{remark}

\section{Some applications in meta-mathematics}

In this section, we discuss some applications of effective inseparability in the meta-mathematics of arithmetic. This section is partly
an exposition of some of the literature, but it also contains original results.
In Section 3.1, we examine some important meta-mathematical properties of theories and the relationship among them. In Section 3.2, we examine
Smory\'{n}ski's theorem about effective inseparability and its application. In Section 3.3, we examine Shoenfield's theorems and their applications to recursively inseparable theories and effectively inseparable theories. In Section 3.4, we show that
there are many ${\sf EI}$ theories weaker than the theory $\mathbf{R}$.

\subsection{Some meta-mathematical properties}

In Definition \ref{varied theories}, we list some important meta-mathematical properties of theories. Then we examine the relationship among these properties. This section is mainly an exposition of some of the literature.

\begin{definition}[Essentially undecidable, creative theory, Rosser theory, Exact Rosser theory]\label{varied theories}
Let $T$ be a consistent RE theory and $(A,B)$ be a disjoint pair of RE sets.
\begin{enumerate}[(1)]
\item We say $T$ is \emph{essentially undecidable} (${\sf EU}$) if any consistent RE extension of $T$ is undecidable.
    \item We say $A\subseteq\mathbb{N}$ is \emph{productive} if there exists a recursive function $f(x)$ (called a productive function for $x$) such that for every number $i$, if $W_i\subseteq A$, then $f(i)\in A- W_i$.
  \item We say $A\subseteq\mathbb{N}$ is \emph{creative} if $A$ is RE and the complement of $A$ is productive.
\item We say $T$ is \emph{creative} if $T_P$ is creative.

\item We say $(A,B)$ is \emph{separable} in $T$ if there is a formula $\phi(x)$ with only one free variable such that if $n\in A$, then $T\vdash \phi(\overline{n})$, and if $n\in B$, then $T\vdash \neg\phi(\overline{n})$.
\item We say $(A,B)$ is \emph{exactly separable} in $T$  if there is a formula $\phi(x)$ with only one free variable such that  $n\in A\Leftrightarrow T\vdash \phi(\overline{n})$, and $n\in B\Leftrightarrow T\vdash \neg\phi(\overline{n})$.
\item We say $T$ is a \emph{Rosser theory} if any disjoint  pair of RE sets is separable in $T$.
    \item We say $T$ is an \emph{exact Rosser theory} if any disjoint  pair of RE sets is exactly separable in $T$.
\end{enumerate}
\end{definition}

\begin{theorem}[Smullyan, Theorem 2, p.221, \cite{Smullyan}]\label{}
For a consistent RE theory $T$, if $T$ is Rosser, then $T$ is ${\sf EI}$.
\end{theorem}
\begin{proof}\label{}
Smullyan proves this fact via Proposition \ref{reduction prop}(2). Here we give a short proof via Theorem \ref{equiv of EI}. Since $T$ is Rosser,  $(T_P, T_R)$ is semi-${\sf DU}$. Thus, by Theorem \ref{equiv of EI}, $T$ is ${\sf EI}$.
\end{proof}

The theory $\mathbf{R}$ was introduced by Tarski, Mostowski and R.\ Robinson  in \cite{Tarski} which is an important base theory in the study of incompleteness and undecidability, and has many nice meta-mathematical properties.

\begin{definition}
Let $\mathbf{R}$ be the theory consisting of schemes $\sf{Ax1}$-$\sf{Ax5}$ with $L(\mathbf{R})=\{\mathbf{0}, \mathbf{S}, +, \cdot, \leq\}$ where  we define $x\leq y$ as $\exists z (z+x=y)$, and  $\overline{n}=\mathbf{S}^n \mathbf{0}$ for $n \in \mathbb{N}$.
\begin{description}
  \item[\sf{Ax1}] $\overline{m}+\overline{n}=\overline{m+n}$;
  \item[\sf{Ax2}] $\overline{m}\cdot\overline{n}=\overline{m\cdot n}$;
  \item[\sf{Ax3}] $\overline{m}\neq\overline{n}$, if $m\neq n$;
  \item[\sf{Ax4}] $\forall x(x\leq \overline{n}\rightarrow x=\overline{0}\vee \cdots \vee x=\overline{n})$;
  \item[\sf{Ax5}] $\forall x(x\leq \overline{n}\vee \overline{n}\leq x)$.
\end{description}
\end{definition}

\begin{definition}[Definable, strongly representable, weakly representable]\label{property of theory}
Let $T$ be a consistent RE theory.
\begin{enumerate}[(1)]
  \item We say a total $n$-ary function $f$ on $\mathbb{N}$ is  \emph{definable}  in $T$ if there exists a $L(T)$-formula $\phi(x_1, \cdots, x_n,y)$ such that for any $\langle a_1, \cdots, a_n\rangle\in\mathbb{N}^n$,
\[T\vdash \forall y [\phi(\overline{a_1}, \cdots, \overline{a_n},y)\leftrightarrow y= \overline{f(a_1, \cdots, a_n)}].\]
We say $\phi(x_1, \cdots, x_n,y)$ \emph{defines} $f$ in $T$.
 \item
We say an $n$-ary relation $R$ on $\mathbb{N}$ is \emph{strongly representable} in $T$ if there exists a $L(T)$-formula $\phi(x_1, \cdots, x_n)$ such that for any $\langle a_1, \cdots, a_n\rangle\in\mathbb{N}^n$, if $R(a_1, \cdots, a_n)$ holds, then
$T\vdash \phi(\overline{a_1}, \cdots, \overline{a_n})$, and if $R(a_1, \cdots, a_n)$ does not hold, then
$T\vdash \neg\phi(\overline{a_1}, \cdots, \overline{a_n})$. We say  $\phi(x_1, \cdots, x_n)$ \emph{strongly represents} the relation $R$.
\item
We say an $n$-ary relation $R$ on $\mathbb{N}$ is \emph{weakly representable} in $T$ if there exists a $L(T)$-formula $\phi(x_1, \cdots, x_n)$ such that for any $\langle a_1, \cdots, a_n\rangle\in\mathbb{N}^n$, $R(a_1, \cdots, a_n)$ holds if and only if
$T\vdash \phi(\overline{a_1}, \cdots, \overline{a_n})$. We say  $\phi(x_1, \cdots, x_n)$ \emph{weakly represents} the relation $R$.

\end{enumerate}
\end{definition}

In this paper, we use the following nice properties of $\mathbf{R}$:

\begin{fact}[\cite{Per 97, Tarski}]~
\begin{itemize}
  \item All recursive  functions are definable in $\mathbf{R}$.
  \item Any disjoint pair of RE sets is separable (in fact exactly separable) in $\mathbf{R}$. Thus,  $\mathbf{R}$ is a Rosser theory (in fact an exact Rosser theory).
\end{itemize}
\end{fact}

As a corollary, if $T$ is a consistent RE extension of $\mathbf{R}$, then $T$ is a Rosser theory, and hence is ${\sf EI}$.
Fact \ref{important fact} provides us with some sufficient conditions to show that a theory is essentially undecidable and recursively inseparable.

\begin{fact}[\cite{Shoenfield61, Smullyan63}]\label{important fact}
Let $T$ be a consistent RE theory.
\begin{enumerate}[(1)]
  \item If $T$ satisfies any one of the following conditions\footnote{Condition $A$ is stronger than condition $B$.}, then $T$ is essentially undecidable:
\begin{enumerate}[(A)]
  \item All recursive functions are definable in $T$.
  \item All recursive sets are strongly representable in $T$.
\end{enumerate}
  \item If all recursive sets are strongly representable in $T$, then $T$ is ${\sf RI}$.
\end{enumerate}
\end{fact}

\begin{remark}
It is easy to check that: (1) if $T$ is ${\sf EI}$, then $T$ is ${\sf RI}$; (2) if $T$ is ${\sf RI}$, then $T$ is ${\sf EU}$; (3) if $T$ is ${\sf EI}$, then $T$ is creative (for any disjoint pair $(A,B)$ of RE sets, if $(A,B)$ is ${\sf EI}$, then both $A$ and $B$ are creative).
\end{remark}

It has been shown that if $T$ satisfies Condition $(A)$ in Fact \ref{important fact}  plus some natural additional condition, then $T$ is creative (e.g., Theorem \ref{group 1}). Theorem \ref{group 1}-\ref{group 3} provide us with some sufficient conditions to show that a theory is creative, ${\sf EI}$ and exact Rosser.

\begin{theorem}[Ehrenfeucht and Feferman, \cite{Ehrenfeucht}]\label{group 1}
Suppose $T$ is a consistent RE theory and has
a formula $x \leq y$ with two free variables $x$ and $y$, satisfying the following conditions\footnote{Shoefield shows in \cite{Shoenfield} that conditions (1)-(3) can be replaced by weaker ones.}:
\begin{enumerate}[(1)]
  \item For each $n\in\mathbb{N}$, $T\vdash \forall x (x \leq \overline{n} \leftrightarrow x = \overline{0} \vee x = \overline{1} \vee\cdots \vee x = \overline{n})$.
  \item For each $n \in\mathbb{N}$, $T\vdash \forall x (x \leq \overline{n} \vee \overline{n} \leq x)$.
  \item All recursive  functions  are definable in $T$.
\end{enumerate}
Then any RE set is weakly representable in $T$ and hence $T$ is creative.
\end{theorem}

\begin{definition}\label{}
Let $T$ be a consistent RE theory.
\begin{enumerate}[(1)]
  \item We say that all recursive sets are \emph{uniformly weakly representable} in $T$ if there is a  recursive function $f(x)$ such that for every number $i$, if $W_i$ is recursive, then $f(i)$ is the G\"{o}del number of a formula of $T$ which weakly represents $W_i$.
  \item
We say all recursive sets are \emph{uniformly strongly representable}  in $T$ if there is a recursive function $g(x, y)$ such that for all numbers $i$ and $j$, if $W_j$ is the complement of $W_i$, then $g(i,j)$ is the G\"{o}del number of a formula which strongly represents $W_i$ in $T$.
\end{enumerate}
\end{definition}

\begin{theorem}[Smullyan, \cite{Smullyan61}]\label{group 2}
Let $T$ be a consistent RE theory.
\begin{enumerate}[(1)]
  \item If all recursive sets are uniformly weakly representable in $T$, then $T$ is creative.
  \item If $T$ is a consistent theory in which all recursive sets are uniformly strongly representable, then $T$ is ${\sf EI}$.
      \end{enumerate}
\end{theorem}

\begin{theorem}[Putnam and Smullyan, \cite{Putnam}]~\label{group 3}
Let $T$ be a consistent RE theory.
\begin{enumerate}[(1)]
  \item  If all recursive functions are definable in $T$ and some ${\sf EI}$  pair of RE sets is separable in $T$, then $T$ is an exact Rosser theory (see Theorem 3, \cite{Putnam}).
  \item If $T$ is a Rosser theory in which  all recursive functions are definable, then $T$ is an exact Rosser theory (see Theorem 4, \cite{Putnam}).
\end{enumerate}
\end{theorem}

A natural question is: if a consistent RE theory $T$ satisfies $(A)$ or $(B)$ in Fact \ref{important fact}, is it creative? Shoenfield answers this question negatively (see Theorem \ref{Shoefield theory}).

\begin{theorem}[Shoenfield, \cite{Shoenfield61}]\label{Shoefield theory}
There is a theory $T$ in which any recursive function is definable but $T$ is not creative, and no non-recursive set is weakly representable. In fact, $T$ has Turing degree $<0^{\prime}$.
\end{theorem}

\begin{definition}\label{}
Given theories $S$ and $T$, we say $S$ is \emph{interpretable} in $T$ if there is a mapping from formulas in the language of S to formulas in the language of $T$ such that axioms of $S$ are provable in $T$ under this mapping (for the precise definition, see \cite{Metamathematics of First-Order Arithmetic}). We demand that this mapping commutes with the propositional connectives.
\end{definition}

From Theorem \ref{Shoefield theory}, we have some important corollaries:
\begin{corollary}
The theory $T$ in Theorem \ref{Shoefield theory} has the following properties:
\begin{itemize}
  \item It is not true that any RE set is weakly representable in $T$;
  \item $T$ is ${\sf RI}$, but is not ${\sf EI}$;
  \item $\mathbf{R}$ is not interpretable in $T$.
\end{itemize}
Moreover, we have:
\begin{itemize}
  \item ``Any recursive function is definable in $T$" does not imply ``$T$ is creative".
  \item ``Any recursive function is definable in $T$" does not imply ``any RE set is weakly representable in $T$".
  \item ``$T$ is ${\sf RI}$" does not imply ``$T$ is creative".
  \item ``$T$ is ${\sf RI}$" does not imply ``$T$ is ${\sf EI}$".
  \item ``Any recursive function is definable in $T$" does not imply  ``$\mathbf{R}$ is interpretable in $T$".
\end{itemize}
\end{corollary}

\begin{remark}
It was an open question that ``whether any recursive function is definable in $T$ implies that $\mathbf{R}$ is interpretable in $T$". Emil Je\v{r}\'{a}bek answered this question negatively and  gave a counterexample in \cite{Emil}. In fact, Theorem \ref{Shoefield theory} provides us with a new counterexample.
\end{remark}

\subsection{Smory\'{n}ski's theorem and its application}

In this section, we examine Smory\'{n}ski's theorem about effective inseparability and its application.

\begin{theorem}[Smory\'{n}ski, \cite{Smorynski}]\label{two EI pair}
Let $(A,C)$ and $(B,D)$ be pairs of effective inseparable RE sets with $A\subseteq B$ and $C\subseteq D$. Then there is a recursive function $f$ such that for any $x$, we have:
\[x\in A\Leftrightarrow f(x)\in A\Leftrightarrow f(x)\in B;\]
\[x\in C\Leftrightarrow f(x)\in C\Leftrightarrow f(x)\in D.\]
\end{theorem}

We show that  the condition ``$A\subseteq B$ and $C\subseteq D$" in Theorem \ref{two EI pair} is necessary. Before the proof, we first introduce some definitions and facts.

\begin{definition}\label{}
We say $\mathbf{Prov}_T(x)$ is a standard provability predicate if it satisfies the following properties:
\begin{description}
  \item[D1] If $T\vdash\phi$, then $T\vdash \mathbf{Prov}_T(\overline{\ulcorner\phi\urcorner})$.
  \item[D2] $T\vdash \mathbf{Prov}_T(\overline{\ulcorner\phi\rightarrow\psi\urcorner})\rightarrow (\mathbf{Prov}_T(\overline{\ulcorner\phi\urcorner})
      \rightarrow \mathbf{Prov}_T(\overline{\ulcorner\psi\urcorner}))$.
  \item[D3] $T\vdash \mathbf{Prov}_T(\overline{\ulcorner\phi\urcorner})\rightarrow \mathbf{Prov}_T(\overline{\ulcorner\mathbf{Prov}_T(\overline{\ulcorner\phi\urcorner})\urcorner})$.
\end{description}
\end{definition}

\begin{theorem}[L\"{o}b's theorem]\label{}
Let $T$ be consistent RE theory. For any sentence $\phi$ and  standard provability predicate $\mathbf{Prov}_T(x)$, we have $T\vdash \mathbf{Prov}_T(\overline{\ulcorner\phi\urcorner})\rightarrow\phi$ if and only if  $T\vdash\phi$.
\end{theorem}

\begin{fact}[Folklore]\label{EI ex}
For any formula $\theta(x)$, the pair $(Fix_T(\theta), Fix_T(\neg\theta))$ is ${\sf EI}$ where $Fix_T(\theta)=\{\phi: T\vdash \phi\leftrightarrow \theta(\overline{\ulcorner\phi\urcorner})\}$.
\end{fact}

\begin{theorem}\label{}
The condition ``$A\subseteq B$ and $C\subseteq D$" in Theorem \ref{two EI pair} cannot be dropped.
\end{theorem}
\begin{proof}\label{}

Suppose that the condition ``$A\subseteq B$ and $C\subseteq D$" in Theorem \ref{two EI pair} can be dropped. Let $T$ be a consistent RE extension of $\mathbf{R}$. Note that  $(T_P, T_R)$ is ${\sf EI}$, and  $(Fix_T(\theta), Fix_T(\neg\theta))$ is ${\sf EI}$ for any formula $\theta(x)$ (by Fact \ref{EI ex}).
Apply Theorem \ref{two EI pair} to $(T_P, T_R)$ and $(Fix_T(\theta), Fix_T(\neg\theta))$ where $\theta=\neg \mathbf{Prov}_T(x)$.
Then there is a recursive function $f$ such that for any sentence $\phi$, we have:

\begin{align}
T\vdash\phi\Leftrightarrow T\vdash f(\phi)\Leftrightarrow T\vdash f(\phi)\leftrightarrow \neg \mathbf{Prov}_T(\overline{\ulcorner f(\phi)\urcorner});\label{first line}\\
T\vdash\neg\phi\Leftrightarrow T\vdash \neg f(\phi)\Leftrightarrow T\vdash f(\phi)\leftrightarrow  \mathbf{Prov}_T(\overline{\ulcorner f(\phi)\urcorner})\label{second line}.
\end{align}

By L\"{o}b's theorem, we have $T\vdash f(\phi)\leftrightarrow  \mathbf{Prov}_T(\overline{\ulcorner f(\phi)\urcorner})$ if and only if   $T\vdash f(\phi)$. By (\ref{second line}), we have $T\vdash\neg\phi\Leftrightarrow T\vdash f(\phi)$.
By (\ref{first line}) and (\ref{second line}), we have $T\vdash\phi\Leftrightarrow T\vdash\neg\phi$ for any sentence $\phi$, which leads to a contradiction.

\end{proof}

\begin{theorem}\label{}

Let $T$ be a Rosser  theory. Let $(A,B)$ be any ${\sf EI}$ pair of RE sets. Suppose $(A,B)$ is separable in $T$ by the formula $\phi(x)$. Define $f: n\mapsto \ulcorner\phi(\overline{n})\urcorner$. Then there is a recursive function $g$ such that for any $n\in\omega$:
\begin{enumerate}[(1)]
  \item $n\in f[A]\Leftrightarrow g(n)\in f[A]\Leftrightarrow g(n)\in T_P$;
  \item $n\in f[B]\Leftrightarrow g(n)\in f[B]\Leftrightarrow g(n)\in T_R$.
\end{enumerate}
\end{theorem}
\begin{proof}\label{}
Note that $f$ is recursive, and if $n\in A$, then $f(n)\in T_P$, and if $n\in B$, then $f(n)\in T_R$. Since $(A,B)$ is semi-reducible to $(f[A],f[B])$ and $(A,B)$ is ${\sf EI}$, by Proposition \ref{reduction prop}, $(f[A],f[B])$ is ${\sf EI}$. Since $T$ is a Rosser  theory, $(T_P,T_R)$ is ${\sf EI}$. Note that $f[A]\subseteq T_P$ and $f[B]\subseteq T_R$. Apply Theorem \ref{two EI pair} to  $(f[A],f[B])$ and $(T_P,T_R)$. Then there is a recursive function $g$ such that for any $n\in\omega$, we have:
\[n\in f[A]\Leftrightarrow g(n)\in f[A]\Leftrightarrow g(n)\in T_P;\]
\[n\in f[B]\Leftrightarrow g(n)\in f[B]\Leftrightarrow g(n)\in T_R.\]
\end{proof}

\subsection{Shoenfield's theorems and their applications}

Shoenfield's theorems in \cite{Shoenfield61} are an important tool in the proof of Theorem \ref{RI COM}. In Theorem \ref{Shoenfield thm1} and Theorem \ref{Shoenfield thm2}, we give detailed proofs of Shoenfield's theorems.
We say that a set $D$ separates $B$ and $C$ if $B\subseteq D$ and $D\cap C=\emptyset$.

\begin{theorem}[Shoenfield, \cite{Shoenfield}]\label{Shoenfield thm1}
For any RE set $A$, there is a disjoint pair $(B, C)$ of RE sets such that $B, C\leq_T \, A$, and for any RE set $D$ that separates $B$ and $C$,
we have $A\leq_T \, D$.
If $A$ is a non-recursive RE set, then $(B, C)$ is  recursively inseparable.
\end{theorem}
\begin{proof}\label{}
Suppose $A=W_e$. Define
\[x\in B\Leftrightarrow \exists y [T_1(e, (x)_0, y)\wedge\forall z\leq y \neg T_1((x)_1, x, z)]\] and
\[x\in C\Leftrightarrow \exists y [T_1(e, (x)_0, y)\wedge\exists z\leq y \, T_1((x)_1, x, z)].\]
We show that $(B, C)$ has the stated properties.

Note that $B$ and $C$ are disjoint RE sets. If $(x)_0 \notin A$, then $x\notin B$ and $x \notin C$; if $(x)_0 \in A$, then we can decide either $x\in B$ or $x \in C$. Thus, we have $B,C \leq_T A$.
\smallskip

Suppose $D$ is a RE set with index $n$,  and $B\subseteq D$ and $D\cap C=\emptyset$. We show that $A\leq_T D$.
\begin{claim}~
$x\in  A \Leftrightarrow \exists z (T_1(n, \langle x, n\rangle, z)\wedge \exists y<z \, T_1(e,x,y))$.
\end{claim}
\begin{proof}\label{}
The right-to-left direction  is obvious. Now we show the left-to-right direction. Suppose $x\in A$. Then either $\langle x,n\rangle\in B$ or $\langle x,n\rangle\in C$. Suppose $\langle x,n\rangle\in C$. Let $y$ be the unique witness such that $T_1(e, x,y)$ holds. Then there exists  $z\leq y$ such that $T_1(n, \langle x,n\rangle, z)$. Then $\langle x,n\rangle\in D$, which contradicts that $D\cap C=\emptyset$. Thus we have $\langle x,n\rangle\in B$.

Let $y$ be the unique witness such that $T_1(e, x, y)$ holds.  Since $\langle x,n\rangle\in D$, we have $T_1(n,\langle x,n\rangle, z)$ holds for some $z$. Since for all $z\leq y$, $\neg T_1(n, \langle x,n\rangle, z)$ holds, we have $z>y$. Thus, we have $\exists z (T_1(n, \langle x,n\rangle, z)\wedge \exists y<z \, T_1(e,x,y))$.
This ends the proof of the claim.
\end{proof}

Now we show that $A\leq_T D$. If $\langle x,n\rangle\notin D$, then $x\notin A$. If $\langle x,n\rangle\in D$, from the above claim, we can effectively decide whether $x\in A$.

We show if $A$ is non-recursive, then ($B, C)$ is recursively inseparable: if $X$ is a recursive set separating $B$ and $C$, then $A\leq_T X$ and hence $A$ is recursive, that leads to a contradiction.
\end{proof}

\begin{theorem}[Shoenfield, \cite{Shoenfield}]\label{Shoenfield thm2}
For any RE set $A$, there is a consistent RE theory $T$ having one non-logical symbol that has the same Turing degree as $A$, and $T$ is essentially undecidable if $A$ is not recursive.
\end{theorem}
\begin{proof}\label{}
Pick the  pair $\langle B, C\rangle$ of RE sets from $A$ as constructed in Theorem \ref{Shoenfield thm1}. Now we define the theory $T_{(B,C)}$ with $L(T_{(B,C)})=\{R\}$ where $R$ is a binary relation symbol.

Let $\Phi_n$ be the statement that there exists an equivalence class of $R$ of size precisely  $n+1$. Let $\Psi_n$ be the statement that
there is at most one equivalence class of $R$ of size precisely  $n$. Let $\Upsilon_n$ be the statement that there are at least $n$ equivalence class of $R$ with at least   $n$ elements\footnote{As Visser pointed out, we include this axiom just to make the proof of Lemma \ref{Janiczak's Lemma} more easy.}.

We denote  the theory $T_{(B,C)}$ by $T$, which contains the following axioms:
\begin{enumerate}[(1)]
  \item the axiom asserting that $R$ is an equivalence relation;
\item $\Psi_n$ for each $n\in\omega$;
\item $\Upsilon_n$ for each $n\in\omega$;
  \item $\Phi_n$ for all $n\in B$;
  \item $\neg\Phi_n$ for all $n\in C$.
\end{enumerate}

Clearly, $T$ is a consistent RE theory. Since $\Phi_n$ is
provable iff $n\in B$, and $\neg\Phi_n$ is provable iff $n \in C$, we have
$B$ and $C$ are recursive in $T$.\smallskip

Lemma \ref{Janiczak's Lemma} is a reformulation of Janiczak's Lemma 2 in \cite{Janiczak} in the context of the theory $T$. Janiczak's Lemma is proved by means of a method known as the elimination of quantifiers.

\begin{lemma}[Janiczak, Lemma 2 in \cite{Janiczak}]\label{Janiczak's Lemma}
Over $T$, every sentence is equivalent to a boolean combination of the $\Phi_n$, and  this boolean combination  can be found explicitly from the given sentence.
\end{lemma}

By Lemma \ref{Janiczak's Lemma}, $T$ is
recursive in $B$ and $C$. Hence $T$ is recursive in $A$ since $B, C\leq_T A$. Since $B$ separates $B$ and $C$,  by Theorem \ref{Shoenfield thm1}, $A\leq_T \, B$. Thus, since $B$ is recursive in $T$, $T$ has the same Turing degree as $A$.
\smallskip

By a standard argument, we can  show that $T$ is essentially undecidable if $A$ is not recursive.
\end{proof}

\begin{remark}
The proof of Theorem \ref{Shoenfield thm2} essentially shows that: given a disjoint pair $(B,C)$ of RE sets, there is a RE theory $T_{(B,C)}$ such that $T_{(B,C)}\equiv_T (B,C)$ and if $(B,C)$ is recursively inseparable, then $T_{(B,C)}$ is essentially undecidable.
\end{remark}

\begin{definition}[\cite{Rogers}]~\label{}
\begin{enumerate}[(1)]
  \item Given $A,B\subseteq \mathbb{N}$, we say $A$ is \emph{$m$-reducible} to $B$ (denoted by $A\leq_m B$) if there is a recursive function $f$ such that $n\in A\Leftrightarrow f(n)\in  B$.
  \item We say $X\subseteq\mathbb{N}$ is $\Sigma^0_n$-complete  if $X$ is a $\Sigma^0_n$ set and for any $\Sigma^0_n$ set $Y$, $Y\leq_m X$.  Similarly for $\Pi^0_n$-complete.
\end{enumerate}
\end{definition}

Now, we discuss some applications of Theorem \ref{Shoenfield thm1} and Theorem \ref{Shoenfield thm2}. We first show that the set of indexes of recursively inseparable theories is $\Pi^0_3$-complete. We use the following result from \cite{Rogers}.

\begin{fact}[Theorem XVI, p. 327, \cite{Rogers}]~\label{R index}
The set $\{e: W_e$ is recursive\} is $\Sigma^0_3$-complete.
\end{fact}

\begin{theorem}\label{RI COM}
The set $\{e: W_e$ is a ${\sf RI}$ theory\} is $\Pi^0_3$-complete.
\end{theorem}
\begin{proof}\label{}
Define $U_0=\{e: W_e$ is undecidable\} and $U_1=\{e: W_e$ is a ${\sf RI}$ theory\}.
We show that there is a recursive function $s$ such that $U_0\leq_m U_1$ via the function $s$.

Note that the construction of the pair $(B,C)$ from a RE set $A$  in Theorem \ref{Shoenfield thm1} is effective: there are recursive functions $f_0$ and $f_1$ such that if $A=W_e$, then $B=W_{f_0(e)}$ and $C=W_{f_1(e)}$ such that $(B,C)$ has the properties as in Theorem \ref{Shoenfield thm1}.
Given a disjoint pair $(B,C)$ of RE sets, note that the construction of $T_{(B,C)}$ in Theorem \ref{Shoenfield thm2} is also effective:  there is a recursive function $h(x,y)$ such that if $B=W_n$ and $C=W_m$ and $B\cap C=\emptyset$, then $T_{(B,C)}=W_{h(n,m)}$ such that $T_{(B,C)}$ has the same Turing degree as $A$.

Define $s(n)=h(f_0(n), f_1(n))$. Note that $s$ is recursive. Suppose $A=W_e$ and $B=W_{f_0(e)}$ and $C=W_{f_1(e)}$ are constructed from $A$ as in Theorem \ref{Shoenfield thm1}. Then $s(e)$ is the index of the theory $T_{(B,C)}$ constructed from $(B,C)$ as in Theorem \ref{Shoenfield thm2}.
We show that $e\in U_0\Leftrightarrow s(e)\in U_1$.

Case one: $e\in U_0$. Note that $(B,C)$ is recursively inseparable since $A=W_e$ is not recursive. We  denote the theory $T_{(B,C)}$ by $T$.
Define $g: n\mapsto \ulcorner \Phi_n\urcorner$. Note that $g$ is recursive, if $n\in B$, then $g(n)\in T_P$, and if $n\in C$, then $g(n)\in T_R$.
We show that $T$ is ${\sf RI}$. Suppose $T$ is not ${\sf RI}$. I.e., there is a recursive set $X$ such that $T_P\subseteq X$ and $X\cap T_R=\emptyset$. Note that $B\subseteq g^{-1}[T_P]\subseteq g^{-1}[X]$ and $C\subseteq g^{-1}[T_R]\subseteq g^{-1}[\overline{X}]=\overline{g^{-1}[X]}$. Since $g$ and $X$ are recursive, $g^{-1}[X]$ is recursive. Thus, $g^{-1}[X]$ is a recursive set separating $B$ and $C$, which  contradicts that $(B,C)$ is ${\sf RI}$.
Thus, $s(e)\in U_1$.

Case two: $e\notin U_0$. Since $T_{(B,C)}$ has the same Turing degree as $A=W_e$, $T_{(B,C)}$ is recursive. Thus, $s(e)\notin U_1$. Hence, $U_0\leq_m U_1$ via the recursive function $s$. By Fact \ref{R index}, $U_0$ is $\Pi^0_3$-complete.
It is easy to check that $U_1$ is $\Pi^0_3$.  Hence, $U_1$ is $\Pi^0_3$-complete.
\end{proof}

As a corollary of Theorem \ref{RI COM}, the set $\{e: (W_{e_0}, W_{e_1})$ is ${\sf RI}$\} is $\Pi^0_3$-complete.

A natural question is whether Theorem \ref{Shoenfield thm1} and Theorem \ref{Shoenfield thm2} can be generalized in the following sense: given a non-recursive RE set $A$, is there an ${\sf EI}$ pair $(B,C)$ such that $A, B$ and $C$ have the same Turing degree? or is there an ${\sf EI}$ theory $T$ such that $T$ has the same Turing degree as $A$?
Let $A$ be a non-recursive RE set. If $(B,C)$ is ${\sf EI}$, then both $B$ and $C$ have the Turing degree $0^{\prime}$. Thus, if $A$ has the Turing degree less  than $0^{\prime}$, then there is no such an ${\sf EI}$ pair (and there is no such an ${\sf EI}$ theory) with the same Turing degree as $A$. If $A$ has the Turing degree $0^{\prime}$, then any ${\sf EI}$ pair  (and any ${\sf EI}$ theory) has the same Turing degree as $A$.

Now, we show that the set of indexes of effectively inseparable theories is $\Sigma^0_3$ using ${\sf EI}\Leftrightarrow {\sf DG}$.

\begin{theorem}
The set $\{e: W_e$ is an ${\sf EI}$ theory\} is $\Sigma^0_3$.
\end{theorem}
\begin{proof}\label{}
Define $U=\{e: W_e$ is an ${\sf EI}$ theory\}.
A direct computation from the definition of ${\sf EI}$ shows that $U$ is $\Sigma^0_4$. Using that ${\sf EI}\Leftrightarrow {\sf DG}$, we could show that $U$ in fact is $\Sigma^0_3$.

By the s-m-n theorem, there is a recursive function $h$ such that if $T=W_e$, then $W_{h(e)}=T_R$.
Note that $``e\in U"$ is equivalent to the following formula:
\[\exists n \forall i \forall j [W_i\cap W_j=\emptyset\rightarrow ((\phi_n(i,j)\in W_e\leftrightarrow \phi_n(i,j)\in W_i)\wedge (\phi_n(i,j)\in W_{h(e)}\leftrightarrow \phi_n(i,j)\in W_j))].\]
Since $e\in U$ can be written in the form $\exists \forall \forall [\Pi^0_1\rightarrow ((\Sigma^0_1\rightarrow \Sigma^0_1)\wedge (\Sigma^0_1\rightarrow \Sigma^0_1))]$, $U$ is $\Sigma^0_3$.
\end{proof}

\begin{remark}
We know that $\{e: W_e$ is creative\} is $\Sigma^0_3$-complete (see \cite{Rogers}). We conjecture that $\{e: W_e$ is an ${\sf EI}$ theory\} is $\Sigma^0_3$-complete.
\end{remark}

\subsection{There are many ${\sf EI}$ theories weaker than the theory $\mathbf{R}$}

In this section, we show that
there are many ${\sf EI}$ theories weaker than the theory $\mathbf{R}$.

\begin{definition}\label{}
For RE theories $S$ and $T$, we use $S\unlhd T$  to denote that $S$ is interpretable in $T$, and $S\lhd T$ to denote that  $S$ is interpretable in $T$ but $T$ is not interpretable in $S$.
\end{definition}

\begin{theorem}[Smullyan, Theorem 4, \cite{Smullyan}]\label{}
For consistent theories $T_1$ and  $T_2$, if $T_1\unlhd T_2$ and $T_1$ is ${\sf RI} ({\sf EI})$, then $T_2$ is ${\sf RI} ({\sf EI})$.
\end{theorem}

\begin{corollary}
If the theory $\mathbf{R}$ is interpretable in $T$, then $T$ is ${\sf EI}$.
\end{corollary}

A natural question is: if $T$ is ${\sf RI} ({\sf EI})$, is the theory $\mathbf{R}$ interpretable in $T$?
We answer these questions negatively.

\begin{definition}\label{pair theory}
Given any disjoint pair $(A,B)$ of RE sets, we construct the theory $T_{(A,B)}$ as follows. Let $L(T_{(A,B)})=\{\mathbf{0}, \mathbf{S}, P\}$. The axioms of $T_{(A,B)}$ consist of:
\begin{enumerate}[(1)]
  \item $\overline{m}\neq \overline{n}$ if $m\neq n$;
  \item $P(\overline{n})$ if $n\in A$;
  \item $\neg P(\overline{n})$ if $n\in B$.
\end{enumerate}
\end{definition}

\begin{theorem}[Cheng, \cite{Cheng}]\label{my thm}
If $(A,B)$ is ${\sf RI}$, then $T_{(A,B)}$ is essentially undecidable and $T_{(A,B)}\lhd \mathbf{R}$.
\end{theorem}

Theorem \ref{RI not R} shows that from any ${\sf RI}$ pair, we can effectively construct a ${\sf RI}$ theory which is strictly weaker than the theory $\mathbf{R}$ w.r.t. interpretation.

\begin{theorem}\label{RI not R}
If $(A,B)$ is ${\sf RI}$, then there is a RE theory $T_{(A,B)}$ such that $T_{(A,B)}$ is ${\sf RI}$ and $T_{(A,B)}\lhd \mathbf{R}$.
\end{theorem}
\begin{proof}\label{}
Let $T_{(A,B)}$ be the theory constructed as in Definition \ref{pair theory}, and we denote it by $T$.
Define $f: n\mapsto \ulcorner P(\overline{n})\urcorner$. Note that $f$ is recursive, if $n\in A$, then $f(n)\in T_P$, and if $n\in B$, then $f(n)\in T_R$.

By Theorem \ref{my thm}, $T\lhd \mathbf{R}$. By a similar argument as in Theorem \ref{RI COM}, we can show that $T$ is ${\sf RI}$: if $X$ is a recursive set separating $T_P$ and $T_R$, then $f^{-1}[X]$ is a recursive set separating $A$ and $B$.

\end{proof}

Theorem \ref{EI not R} shows that from any ${\sf EI}$ pair, we can effectively construct an ${\sf EI}$ theory which is strictly weaker than the theory $\mathbf{R}$ w.r.t. interpretation.

\begin{theorem}\label{EI not R}
If $(A,B)$ is ${\sf EI}$, then there is a RE theory $T_{(A,B)}$ such that $T_{(A,B)}$ is ${\sf EI}$ and $T_{(A,B)}\lhd \mathbf{R}$.
\end{theorem}
\begin{proof}\label{}
Let $T_{(A,B)}$ be the theory constructed as in Definition \ref{pair theory}, and we denote it by $T$. Define the recursive function $f: n\mapsto \ulcorner P(\overline{n})\urcorner$.

By Theorem \ref{my thm}, $T\lhd \mathbf{R}$. Now, we show that $T$ is ${\sf EI}$. Suppose $(A,B)$ is ${\sf EI}$ via the recursive function $h$. By the s-m-n theorem, there is a recursive function $g$ such that for any $i\in\omega$, $f^{-1}[W_i]=W_{g(i)}$.

Define $s(i,j)=f(h(g(i), g(j)))$. We show that $T$ is ${\sf EI}$ via the recursive function $s$.

Suppose $T_P\subseteq W_i$, $T_R\subseteq W_j$ and $W_i\cap W_j=\emptyset$.
Note that $A\subseteq f^{-1}[T_P]\subseteq f^{-1}[W_i]=W_{g(i)}, B\subseteq f^{-1}[T_R]\subseteq f^{-1}[W_j]=W_{g(j)}$, and $W_{g(i)}\cap W_{g(j)}=\emptyset$. Since $(A,B)$ is ${\sf EI}$ via the recursive function $h$, $h(g(i), g(j))\notin W_{g(i)}\cup W_{g(j)}=f^{-1}[W_i]\cup f^{-1}[W_j]$. Thus, $s(i,j)\notin W_i\cup W_j$. I.e. $T$ is ${\sf EI}$ via the recursive function $s$.
\end{proof}

We conclude the paper with one question for future research, which we did not explore. 
\begin{question}\label{}
Let $T$ be a consistent RE theory. If $T$ is a Rosser theory, is $T$ an exact Rosser theory?

\end{question}

{}

\end{document}